\documentclass[11pt]{amsart}
\usepackage[sorted-cites]{amsrefs}
\usepackage{verbatim}
\usepackage{color}
\usepackage{amscd}
\renewcommand{\bar}{\overline}

\newfont{\fnt}{cmr10 scaled 550}

\newtheorem{theorem}{Theorem}

\newtheorem{lemma}{Lemma}
\newtheorem{cor}{Corollary}
\newtheorem{prop}{Proposition}

\newtheorem{definition}{Definition}

\newtheorem{example}{Example}

\font\strange=msbm10

\renewcommand{\epsilon}{\varepsilon}

\renewcommand{\Sigma}{\varSigma}
\newcommand{\R}{{{\mathchoice  {\hbox{$\textstyle{\text{\strange R}}$}}
{\hbox{$\textstyle{\text{\strange R}}$}}
{\hbox{$\scriptstyle  N\kern-0.3em  R$}}
{\hbox{$\scriptscriptstyle  R\kern-0.2em  R$}}}}}

\newcommand{\Z}{{{\mathchoice  {\hbox{$\textstyle{\text{\strange Z}}$}}
{\hbox{$\textstyle{\text{\strange Z}}$}}
{\hbox{$\scriptstyle  Z\kern-0.3em  Z$}}
{\hbox{$\scriptscriptstyle  Z\kern-0.2em  Z$}}}}}

\newcommand{\N}{{{\mathchoice  {\hbox{$\textstyle{\text{\strange N}}$}}
{\hbox{$\textstyle{\text{\strange N}}$}}
{\hbox{$\scriptstyle  N\kern-0.3em  N$}}
{\hbox{$\scriptscriptstyle  N\kern-0.2em  N$}}}}}

\renewcommand{\phi}{\varphi}
\usepackage{amsmath,amsthm,amscd}

\begin{document}

\title[Simons' type equation  for $f$-minimal hypersurfaces]{Simons' type equation  for $f$-minimal hypersurfaces and applications}



\thanks{The first and third authors are  partially supported by CNPq and Faperj of Brazil. The second author is supported by CNPq of Brazil.}

\author{Xu Cheng}
\address{Instituto de Matematica e Estatistica, Universidade Federal Fluminense,
Niter\'oi, RJ 24020, Brazil, email: xcheng@impa.br}
\author{Tito Mejia}
\address{Instituto de Matematica e Estatistica, Universidade Federal Fluminense,
Niter\'oi, RJ 24020, Brazil, email: tmejia.uff@gmail.com}
\author[Detang Zhou]{Detang Zhou}
\address{Instituto de Matematica e Estatistica, Universidade Federal Fluminense,
Niter\'oi, RJ 24020, Brazil, email: zhou@impa.br}


\newcommand{\M}{\mathcal M}

\begin{abstract}  We derive the Simons' type equation for $f$-minimal hypersurfaces in weighted Riemannian manifolds and apply it to obtain a pinching theorem for closed $f$-minimal hypersurfaces immersed in  the product manifold $\mathbb{S}^n(\sqrt{2(n-1)})\times \mathbb{R}$ with $f=\frac {t^2}{4}$.  Also we  classify closed $f$-minimal hypersurfaces with $L_f$-index one immersed  in $\mathbb{S}^n(\sqrt{2(n-1)})\times \mathbb{R}$ with the same $f$ as  above.
\end{abstract}

\maketitle
\baselineskip=18pt

\section{Introduction}\label{introduction}

In \cite{S}, Simons  proved an identity, called Simons' equation, for the Laplacian of $|A|^2$, the  square of the norm of the second fundamental form of minimal hypersurfaces in  Riemannian manifolds.  Simons' equation plays an important role  in the study of minimal hypersurfaces. Recently, Colding-Minicozzi  (\cite{CM3}, Section 10.2) derived the Simons' type equation for self-shrinkers for the  mean curvature flow in the Euclidean space $\mathbb{R}^{n+1}$ and applied it to classify complete embedded self-shrinkers of nonnegative mean curvature with polynomial volume growth. Later, using the Simons' type equation for self-shrinkers, Le-Sesum \cite{LS} and Cao-Li \cite{CL}  obtained gap theorems for self-shrinkers.

 Note that  both  minimal hypersurfaces and  self-shrinkers are  special cases of $f$-minimal hypersurfaces in the weighted Riemannian manifolds.
 In this paper, we will prove a Simons' type equation for $f$-minimal hypersurfaces in a smooth metric measure space $(M, \overline{g}, e^{-f}d\mu)$, that is, 
an identity for the weighted Laplacian $\Delta_f$ of $|A|^2$  of  $f$-minimal hypersurfaces, involving the Bakry-\'Emery Ricci curvature $\overline{Ric}_f$ (see Theorem \ref{prop2}).  Also,  we  derive the equations for the weighted Laplacian $\Delta_f$ of some other geometric quantities on $f$-minimal hypersurfaces, like the mean curvature $H$, etc, see, for instance,  Propositions \ref{prop1} and \ref{aut2}. Since these equations involve  $\overline{Ric}_f$,    we naturally would like to consider the cases in which the ambient manifold are gradient Ricci solitons, that is, $M$  satisfies $\overline{Ric}_f=C\overline{g}$, see Corollaries \ref{cor-H}, \ref{simons-eq-grad} and \ref{cor-alpha}. 
Further,  we  apply the equations mentioned above  to the special case of $f$-minimal hypersurfaces in the cylinder shrinking soliton $\mathbb{S}^n(\sqrt{2(n-1)}\times \mathbb{R}$ with $f=\frac{t^2}{4}$. Namely we  obtain the following pinching theorem.
\begin{theorem}\label{pinching}
Let  $\Sigma^{n}$ be a closed  immersed $f$-minimal  hypersurface in the product manifold $\mathbb{S}^{n}\bigr(\sqrt{2(n-1)}\bigr)\times\mathbb{R}$, $n\geq 3$, with $f=\frac{t^2}{4}$. If the square of the  norm of the second fundamental form of $\Sigma$ satisfies
\begin{equation*}
\frac14\biggr(1-\sqrt{1-\frac{8}{n-1}\alpha^2(1-\alpha^2)}\biggr)\leq |A|^2\leq \frac14\biggr(1+\sqrt{1-\frac{8}{n-1}\alpha^2(1-\alpha^2)}\biggr),
\end{equation*}
 then $\Sigma$ is $\mathbb{S}^{n}(\sqrt{2(n-1)})\times\{0\}$, where $\alpha=\langle\nu,\frac{\partial}{\partial t}\rangle$, $\nu$ is the outward unit normal to $\Sigma$ and $t$ denotes the coordinate of the factor $\mathbb{R}$ of $\mathbb{S}^{n}\bigr(\sqrt{2(n-1)}\bigr)\times\mathbb{R}$.
\end{theorem}

\indent Observe that $n\geq 3$ implies that $\frac{8}{n-1}\alpha^2(1-\alpha^2)\leq 1$ and hence the inequalities in Theorem \ref{pinching} make sense.  Theorem \ref{pinching} implies that 

\begin{cor}  There is no any closed immersed $f$-minimal  hypersurface in the product manifold $\mathbb{S}^{n}\bigr(\sqrt{2(n-1)}\bigr)\times\mathbb{R}$, $n\geq3$, with $f=\frac{t^2}{4}$, so that its squared norm of the second fundamental form  satisfies
\begin{equation*}
\frac14\biggr(1-\sqrt{1-\frac{2}{n-1}}\biggr)\leq |A|^2\leq \frac14\biggr(1+\sqrt{1-\frac{2}{n-1}}\biggr).
\end{equation*}
\end{cor}
 
  Next, we discuss, as another application,  the eigenvalues and the index of the operator $L_f$ on $f$-minimal hypersurfaces.  The eigenvalues of $L$-operator for self-shrinkers were discussed in \cite{CM3} and 
recently,  Hussey \cite{H} studied the index of $L$-operator for self-shrinkers in $\mathbb{R}^{n+1}$. Observe that $L$-operator is just $L_f$ operator for self-shrinkers  (see Example \ref{ex-L}). In this paper, we classify closed $f$-minimal hypersurfaces  in the cylinder shrinking soliton $\mathbb{S}^n(\sqrt{2(n-1)})\times \mathbb{R}$ whose $L_f$ operators have index one and prove that
  \begin{theorem}\label{theo6}
Let $\Sigma^{n}$ be a closed immersed $f$-minimal hypersurface in
the product manifold $\mathbb{S}^{n}(\sqrt{2(n-1)})\times\mathbb{R}$ with $f=\frac{t^2}{4}$. Then $L_f$-$\textrm{ind}(\Sigma)\geq 1$. Moreover the equality holds if and only if $\Sigma=\mathbb{S}^{n}(\sqrt{2(n-1)})\times\{0\}$.
\end{theorem}
  
For complete noncompact $f$-minimal hypersurfaces in cylinder shrinking soliton $\mathbb{S}^n(\sqrt{2(n-1)})\times \mathbb{R}$ with $f=\frac{t^2}{4}$, the first and third authors \cite{CZ2} proved the results corresponding  to Theorems \ref{pinching} and \ref{theo6}. In this case complete noncompact $f$-minimal hypersurfaces are assumed to have polynomial volume growth, which is equivalent to properness of  immersion, or finiteness of  weighted volume (see \cite{XZ} and \cite{CMZ2}).

The rest of this paper is organized as follows: In Section \ref{notation}  some definitions and notations are given; In Section \ref{relation}  we prove the Simons' type equation and the equations for $\Delta_f$ of other geometric quantities; In Section \ref{sec-index}, we calculate the index of $L_f$ operator on closed $f$-minimal hypersurfaces in the cylinder shrinking soliton $\mathbb{S}^n(\sqrt{2(n-1)})\times \mathbb{R}$; In Section \ref{sec-appl}, we prove Theorem \ref{pinching}  for closed $f$-minimal hypersurfaces in the cylinder shrinking soliton $\mathbb{S}^n(\sqrt{2(n-1)})\times \mathbb{R}$.

\section{Definitions and notation}\label{notation}

Let  $(M^{n+1}, \bar{g}, e^{-f}d\mu)$ be a smooth metric measure space, which is an $(n+1)$-dimensional Riemannian manifold $(M^{n+1}, \bar{g})$ together with a weighted volume form $e^{-f}d\mu$ on $M$, where $f$ is a smooth function on $M$ and $d\mu$ the volume element induced by the metric $\bar{g}$.   In this paper, unless otherwise specified, we denote by a bar all quantities on $(M, \bar{g})$, for instance by $\overline{\nabla}$ and $\overline{\textrm{Ric}}$, 
 the Levi-Civita connection  and  the  Ricci curvature tensor of $(M, \bar{g})$ respectively.
For $(M, \bar{g}, e^{-f}d\mu)$, the $\infty$-Bakry-\'Emery Ricci curvature tensor $\overline{\text{Ric}}_{f}$ (for simplicity, Bakry-\'Emery Ricci curvature), which  is defined by
$$
\overline{\text{Ric}}_{f }:=\overline{\text{Ric}}+\overline{\nabla}^{2}f,
$$
where  $\overline{\nabla}^{2}f$ is the Hessian of $f$ on $M$.
If $f$ is constant,   $\overline{\textrm{Ric}}_{f}$ is the  Ricci curvature  $\overline{\textrm{Ric}}$.

Now  let $i: \Sigma^{n}\to M^{n+1}$ be an $n$-dimensional smooth immersion.
Then $i$ induces a metric $g=i^*\bar{g}$ on $\Sigma$ so that $i: (\Sigma^{n}, g) \to (M, \bar{g})$ is an isometric immersion.  We will denote for instance by ${\nabla}$,  $\textrm{Ric}$, $\Delta$ and $d\sigma$, 
 the Levi-Civita connection, the  Ricci curvature tensor,  the Laplacian, and the volume element of $(\Sigma, g)$ respectively.

The restriction of the function $f$ on $\Sigma$, still denoted by $f$,   yields a weighted measure  $e^{-f}d\sigma$ on   $\Sigma$ and hence   an induced smooth metric measure space $(\Sigma^{n}, g, e^{-f}d\sigma)$. The associated weighted Laplacian ${\Delta}_{f}$  on  $\Sigma$ is defined  by
$${\Delta}_{f}u:={\Delta} u-\langle{\nabla} f,{\nabla} u\rangle.$$
 The second order operator ${\Delta}_f$  is a self-adjoint operator on $L^2(e^{-f}d\sigma)$, the space  of square integrable functions on $\Sigma$ with respect to the measure $e^{-f}d\sigma$.

 We define the second fundamental form $A$  of $\Sigma$  by 
$$A: T_p\Sigma\times T_p\Sigma\to \mathbb{R}, \quad  A(X,Y)=-\langle\overline{\nabla}_XY, \nu\rangle, $$
\noindent where $p\in \Sigma, X, Y\in T_p \Sigma$, $\nu$
is a unit normal vector at $p$.

 In a local orthonormal  system $\{e_i\}, i=1,\ldots, n$ of $\Sigma$, the components of $A$ are denoted by $a_{ij}=A(e_i,e_j)=\langle\overline{\nabla}_{e_i}\nu, e_j\rangle$.
The shape operator  $A$ and the mean curvature $H$ of $\Sigma$ are   defined by:  $$A: T_p\Sigma\to T_p\Sigma, AX=\overline{\nabla}_X\nu, X\in T_p\Sigma;$$  $$H=\text{tr}A=\displaystyle\sum_{i=1}^na_{ii}.$$   
With the above notations, we  have the following 
\begin{definition} The weighted mean curvature $H_f$ of the hypersurface $\Sigma$ is defined by  $$H_f=H-\langle \overline{\nabla} f,\nu\rangle.$$ $\Sigma$ is called to be an $f$-minimal hypersurface if it satisfies
\begin{equation} H=\langle \overline{\nabla} f,\nu\rangle.
\end{equation}
\end{definition}

\begin{definition}The weighted volume  of $\Sigma$ is defined by
\begin{equation}V_f(\Sigma):=\int_\Sigma e^{-f}d\sigma.
\end{equation}

\end{definition}
\noindent It is known that $\Sigma$ is $f$-minimal if and only if it is a critical point of the weighted volume functional. On the other hand, we can view it in another manner:
$\Sigma $ is $f$-minimal in $(M,\overline{g})$ is equivalent to that   $(\Sigma, i^*\tilde{g})$ is minimal in $(M,\tilde{g})$, where the  conformal metric $\tilde{g}=e^{-\frac{2f}{n}}\bar {g}$ (cf \cite{CMZ2}).

Now we assume that  $\Sigma$ is a two-sided hypersurface, that is, there  is a globally-defined unit normal $\nu$ on $\Sigma$.

\begin{definition} For  a two-sided hypersurface $\Sigma$,   the $L_f$ operator on $\Sigma$ is given by $$L_f:=\Delta_f+|A|^2+\overline{\text{Ric}}_f(\nu,\nu),$$ where $|A|^2$ denotes the square of the norm of the second fundamental form $A$ of $\Sigma$.  

The operator $L_{f}=\Delta_{f}+|A|^2+\overline{\textrm{Ric}}_{f}(\nu,\nu)$ is called $L_{f}$-stability operator of  $\Sigma$.
\end{definition}

\begin{example}\label{ex-L} For self-shrinkers in $\mathbb{R}^{n+1}$,  the operator $L_f$, where $f=\frac{|x|^2}{4}$, is just  the $L$ operator in \cite{CM3}:
\begin{equation*}
L=\Delta-\frac{1}{2}\langle x,\nabla\cdot\rangle+|A|^{2}+\frac{1}{2}.
\end{equation*}
\end{example}

\begin{definition}A two-sided $f$-minimal hypersurface  $\Sigma$ is said to be $L_{f}$-stable if  for any  compactly supported  smooth function $\phi\in C_o^{\infty}(\Sigma)$, it holds that
\begin{equation}\label{Lf-stab}
-\int_{\Sigma}\phi L_f\phi e^{-f}d\sigma\geq 0,
\end{equation}
or equivalently
\begin{equation}\int_{\Sigma}\big[|\nabla\phi|^2-(|A|^{2}+\overline{\textrm{Ric}}_{f}(\nu,\nu))\phi^2\big]e^{-f}d\sigma\geq 0.
\end{equation}
\end{definition}

\noindent $L_f$-stability  means that the second variation of the weighted volume of $f$-minimal hypersurface $\Sigma$ is non-negative. 
 Further, one has the definition of $L_f$-index of $f$-minimal hypersurfaces. Since $\Delta_f$ is self-adjoint in the weighted space $L^2(e^{-f}d\sigma)$, we may define a symmetric bilinear form on space $C_o^{\infty}(\Sigma)$  of compactly supported smooth functions on $\Sigma$ by, for $\phi, \psi\in C_o^{\infty}(\Sigma)$,
\begin{equation}\begin{split}
B_f(\phi, \psi):&=-\int_\Sigma \phi L_f\psi e^{-f}d\sigma\\
&=\int_{\Sigma}[\langle\nabla\phi,\nabla \psi\rangle-(|A|^{2}+\overline{\textrm{Ric}}_{f}(\nu,\nu))\phi\psi]e^{-f}d\sigma.
\end{split}
\end{equation}
\begin{definition}The $L_{f}$-index of $\Sigma$, denoted by $L_{f}$-ind$(\Sigma)$,  is defined to be the maximum of the dimensions of negative definite  subspaces for $B_f$. 
\end{definition}

\noindent In particular, $\Sigma$ is $L_f$-stable if and only if $L_{f}$-ind$(\Sigma)=0$. 
 $L_f$-index of $\Sigma$ has the following  equivalent definition:
 Consider the Dirichlet eigenvalue problems of $L_f$ on a compact domain $\Omega\subset \Sigma$: 
 $$L_fu=-\lambda u, u\in\Omega;  \quad u|_{\partial \Omega}=0.$$
 $L_{f}$-ind$(\Sigma)$ is defined to be the  supremum over compact domains of $\Sigma$ of the number of negative  (Dirichlet) eigenvalues of $L_f$ (cf \cite{F}).

It is known that an   $f$-minimal hypersurface  $(\Sigma,{g})$ is $L_f$-stable if and only if  $(\Sigma, i^*\tilde{g})$ is stable as a minimal surface in $(M,\tilde{g})$. Further, the Morse index of $L_f$ on $(\Sigma,g)$ is equal to the Morse index of Jacobi  operator on minimal hypersurface $(\Sigma, i^*\tilde{g})$ (see \cite{CMZ2}).

We will take  the following convention for  tensors. For instance, under a local frame field on $M$, suppose that $T=(T_{j_1\ldots j_r})$ is a $(r,0)$-tensor on $M$. 
 The components of the covariant derivative $\overline{\nabla}T$ are denoted by $T_{j_1\ldots j_r;i}$, that is,
$$T_{j_1\ldots j_r;i} =(\overline{\nabla}_{e_i}T)(e_{j_1},\ldots,e_{j_r})=(\overline{\nabla}T)(e_i, e_{j_1},\ldots,e_{j_r}).$$
Meanwhile,  under a local frame field on $\Sigma$, suppose that $S=(S_{k_1\ldots k_s})$ is an $(s,0)$-tensor on $\Sigma$. The components of the covariant derivative ${\nabla}S$ are denoted by $S_{k_1\ldots k_s,l}$, that is,
$$S_{k_1\ldots k_s,l} =({\nabla}_{e_l}S)(e_{k_1},\ldots,e_{k_s})=({\nabla}S)(e_l, e_{k_1},\ldots,e_{k_s}).$$
Throughout this paper,  we assume that the $f$-minimal hypersurfaces are orientable and without boundary. For a closed hypersurface, we choose $\nu$ to be outer unit normal.
In the final of this section, we refer the interested readers to \cite{CMZ1}, \cite{CMZ2}, \cite{E}  and the references therein for  examples and more details about $f$-minimal hypersurfaces.

\section{Simons' type equation for   $f$-minimal hypersurfaces}\label{relation}

First we calculate the weighted Laplacian $\Delta_f$ for  mean curvature $H$ of $f$-minimal hypersurfaces..

\begin{prop}\label{prop1}
Let $(\Sigma^{n}, g)$ be an $f$-minimal hypersurface isometrically immersed in a smooth metric measure space $(M,\bar{g},e^{-f}d\mu)$. Then the  mean curvature $H$ of $\Sigma$ satisfies that 
\begin{eqnarray}\label{f-lap-H}
\Delta_fH
&=&2\textrm{tr}_g(\overline{\nabla}^3f(\cdot,\nu, \cdot)|_{\Sigma})-\text{tr}_g(\overline{\nabla}^3f(\nu, \cdot, \cdot)|_{\Sigma})\\
&&+2\langle A,  \overline{\nabla}^2f|_{\Sigma}\rangle_g -\overline{\textrm{Ric}}_{f}(\nu,\nu)H-|A|^{2}H\nonumber,
\end{eqnarray}
or equivalently
\begin{eqnarray}\label{f-lap-H-1}
\Delta_fH&=&2\sum_{i=1}^{n}(\overline{\nabla}^3f)_{i\nu i}-\sum_{i=1}^n(\overline{\nabla}^3f)_{\nu i i}+2\sum_{i,j=1}^{n}a_{ij}(\overline{\nabla}^2f)_{ij}\\
&&-\overline{\textrm{Ric}}_{f}(\nu,\nu)H-|A|^{2}H,\nonumber
\end{eqnarray}
where $\{e_1,\ldots, e_n\}$ is a local orthonormal frame field on $\Sigma$, $\nu$ denotes the unit normal to $\Sigma$, and $|_{\Sigma}$ denotes the restriction to $\Sigma$.
\end{prop}
 \begin{proof} We  choose  a local  orthonormal frame filed $\{e_i\}_{i=1}^{n+1}$  for $M$ 
so that,   restricted to $\Sigma$,  $\{e_i\}_{i=1}^{n}$ are tangent  to $\Sigma$,  and $e_{n+1}=\nu$ is the unit normal to $\Sigma$.  {\it Throughout this paper,  for simplicity of notation, we substitute $\nu$ for the subscript $n+1$ in the components of the tensors on $M$, for instance $\bar{R}_{\nu ikj}=\bar{Rm}(\nu, e_i, e_k, e_j)$, $(\overline{\nabla}^2f)_{\nu i}=(\overline{\nabla}^2f)(\nu, e_i)$}.
 Differentiating  the mean curvature $H=\langle\overline{\nabla}f,\nu\rangle$, we have, for $1\leq i\leq n$,
\begin{eqnarray}\label{der-H}
e_iH&=&e_i\langle\overline{\nabla}f,\nu\rangle\\
&=&\langle\overline{\nabla}_{e_i}(\overline{\nabla}f),\nu\rangle+\langle\overline{\nabla}f,\overline{\nabla}_{e_i}\nu\rangle\nonumber\\
&=&\overline{\nabla}^2f(\nu, e_i)+\sum_{k=1}^na_{ik}\langle\overline{\nabla}f,e_k\rangle.\nonumber
\end{eqnarray}
Then for $1\leq i, j\leq n$,
\begin{equation}\label{eqs1}
e_je_i(H)=e_j(\overline{\nabla}^2f(\nu, e_i))+\sum_{k=1}^ne_j(a_{ik})f_k+\sum_{k=1}^na_{ik}(e_j\langle\overline{\nabla}f,e_k\rangle).
\end{equation}
For simplicity of calculation. For   an fixed point $p\in\Sigma$, we may further  choose the local orthonormal frame $\{e_1,\ldots,e_n\}$ so that  $\nabla_{e_i}e_j( p)=(\overline{\nabla}_{e_i}e_j)^{\top}(p)=0$, $1\leq i,j\leq n$. 
Then at $p$, for $1\leq i, j\leq n$,
\begin{eqnarray}
&&e_j[\overline{\nabla}^2f(\nu, e_i)]\nonumber \\
&=&\overline{\nabla}_{e_j}[\overline{\nabla}^2f(\nu, e_i)]\nonumber\\
&=&\overline{\nabla}_{e_j}(\overline{\nabla}^2f)(\nu, e_i)+\overline{\nabla}^2f(\overline{\nabla}_{e_j}\nu,e_i)+
\overline{\nabla}^2f(\nu,\overline{\nabla}_{e_j}e_i)\label{H-term1-1}\\
&=&\overline{\nabla}^3f(e_j,\nu,e_i)+\sum_{k=1}^na_{jk}\overline{\nabla}^2f(e_k,e_i)+
\overline{\nabla}^2f(\nu, \langle\overline{\nabla}_{e_j}e_i,\nu)\rangle\nu)\nonumber\\
&=&(\overline{\nabla}^3f)_{j\nu i}+\sum_{k=1}^na_{jk}(\overline{\nabla}^2f)_{ki}-a_{ji}(\overline{\nabla}^2f)_{\nu\nu}.\nonumber
\end{eqnarray}
In the last second equality of (\ref{H-term1-1}), we used the assumption: $\nabla_{e_j}e_i( p)=0$,  $1\leq i,j\leq n$. Also by this assumption and the Codazzi equation, we have that, at $p$ for $1\le i,j\le n, $
\begin{eqnarray*}
e_{j}(a_{ik})&=&a_{ik,j}=a_{ij,k}+\bar{R}_{\nu ikj};\\
e_j\langle\overline{\nabla}f,e_k\rangle &=&\langle\overline{\nabla} _{e_j}(\overline{\nabla}f),e_k\rangle+\langle\overline{\nabla}f, \overline{\nabla} _{e_j}e_k\rangle=(\overline{\nabla}^2f)_{jk}-a_{jk}{f}_\nu;\\
(\nabla^2 H)({e_j}, {e_i})&=&e_je_iH.
\end{eqnarray*}
Substituting these equalities and \eqref{H-term1-1} into \eqref{eqs1}, we have at $p$, for $1\leq i,j\leq n$,
\begin{eqnarray}
(\nabla^2H)(e_j,e_i)&=&(\overline{\nabla}^3f)_{j\nu i}+\sum_{k=1}^na_{jk}(\overline{\nabla}^2f)_{ki}-a_{ji}(\overline{\nabla}^2f)_{\nu\nu}\nonumber\\
&&+\sum_{k=1}^{n}a_{ij,k}f_k+\sum_{k=1}^{n}\bar{R}_{\nu ikj}f_k\label{eq1}\\
&&+\sum_{k=1}^{n}a_{ik}(\overline{\nabla}^2f)_{jk}-\sum_{k=1}^{n}a_{ik}a_{jk}{f}_\nu. \nonumber
\end{eqnarray}
On the other hand,  it holds that on $\Sigma$:
\begin{eqnarray*}
(\overline{\nabla}^3f)_{i\nu j}&=&(\overline{\nabla}^2f)_{\nu j;i}=(\overline{\nabla}^2f)_{j\nu;i}\\
&=&(\overline{\nabla}^2f)_{ji;\nu}+\sum_{k=1}^{n+1}{f}_{k}\bar{R}_{k j\nu i}\\
&=&(\overline{\nabla}^3f)_{\nu ji}+{f}_{\nu}\bar{R}_{\nu i \nu j}+\sum_{k=1}^nf_k\bar{R}_{\nu ikj}.
\end{eqnarray*}
So
\begin{equation}\label{eqs0}
\sum_{k=1}^nf_k\bar{R}_{\nu ikj}=(\overline{\nabla}^3f)_{i\nu j}-(\overline{\nabla}^3f)_{\nu ji}-{f}_{\nu}\bar{R}_{\nu i \nu j}.
\end{equation}
Substituting~\eqref{eqs0} into~\eqref{eq1} and noting $f_{\nu}=H$, we have at $p$, for $1\leq i,j\leq n$,
\begin{eqnarray}\label{Hessian-H} 
(\nabla^2H)(e_j,e_i)&=&(\overline{\nabla}^3f)_{j\nu i}+(\overline{\nabla}^3f)_{i\nu j}-(\overline{\nabla}^3f)_{\nu ji}\\
&&+\sum_{k=1}^na_{jk}(\overline{\nabla}^2f)_{ki}
+\sum_{k=1}^na_{ik}(\overline{\nabla}^2f)_{jk}+\sum_{k=1}^{n}a_{ij,k}f_k\nonumber\\
&&-H\bar{R}_{i\nu j\nu}-a_{ji}(\overline{\nabla}^2f)_{\nu\nu}-\sum_{k=1}^{n}a_{ik}a_{kj}H.\nonumber
\end{eqnarray}
Taking the trace, we have that at $p$,
\begin{eqnarray}\label{eqs2-1}
\Delta H
&=&2\sum_{i=1}^{n}(\overline{\nabla}^3f)_{i\nu i}-\sum_{i=1}^{n}(\overline{\nabla}^3f)_{\nu ii}+2\sum_{i,k=1}^na_{ik}(\overline{\nabla}^2f)_{ki}\\
& &+\langle\nabla f,\nabla H\rangle-\overline{\textrm{Ric}}_{f}(\nu,\nu)H-|A|^{2}H.\nonumber
\end{eqnarray}
Since $p\in\Sigma$ is arbitrary and \eqref{eqs2-1} is independent of the choice of frame, \eqref{eqs2-1} holds on $\Sigma$. By \eqref{eqs2-1} and  $\Delta_f=\Delta-\langle \nabla f,\nabla H\rangle$, we obtain   \eqref{f-lap-H-1} and also equivalent identity \eqref{f-lap-H}.

\end{proof}
Proposition \ref{prop1} yields the following
\begin{cor} \label{cor-H} With the same assumption and notation as in Proposition \ref{prop1}, 
\begin{equation}
L_fH
=2\text{tr}_g(\overline{\nabla}^3f(\cdot,\nu, \cdot)|_{\Sigma})-\text{tr}_g(\overline{\nabla}^3f(\nu, \cdot, \cdot)|_{\Sigma})\\
+2\langle A,  \overline{\nabla}^2f|_{\Sigma}\rangle_g,
\end{equation}
or equivalently
\begin{equation}
L_fH=2\sum_{i=1}^{n}(\overline{\nabla}^3f)_{i\nu i}-\sum_{i=1}^n(\overline{\nabla}^3f)_{\nu i i}+2\sum_{i,j=1}^{n}a_{ij}(\overline{\nabla}^2f)_{ij}.
\end{equation}
\end{cor}
Next we will derive the Simons' type equation for $f$-minimal hypersurfaces.
\begin{theorem}\label{prop2}
Let $(\Sigma^{n},g)$ be an $f$-minimal hypersurface isometrically immersed in $(M,\overline{g},e^{-f}d\mu)$. Then the   square of norm of the second fundamental form of $\Sigma$ satisfies that
\begin{eqnarray}\label{simons-eq}
\frac{1}{2}\Delta_{f}|A|^{2}&=&|\nabla A|^{2}+2\sum_{i,j,k=1}^{n}a_{ij}a_{ik}(\overline{Ric}_f)_{jk}-(\overline{Ric}_f)_{\nu\nu}|A|^{2}-|A|^{4}\nonumber\\
& &+2\sum_{i,j=1}^na_{ij}(\overline{Ric}_f)_{i\nu;j}
-\sum_{i,j=1}^{n}a_{ij}(\overline{Ric}_f)_{ij;\nu}+\sum_{i,j=1}^{n}a_{ij}\bar{R}_{i\nu j\nu;\nu}\\
& &-2\sum_{i,j,k=1}^{n}a_{ij}a_{ik}\bar{R}_{j\nu k\nu}-2\sum_{i,j,k,l=1}^{n}a_{ij}a_{lk}\bar{R}_{iljk},\nonumber
\end{eqnarray}
where the notations are same as in Proposition \ref{prop1}.
\end{theorem}
\begin{proof} Simons \cite{S} proved the following identity (see, e.g.,\cite{SSY} (1.20)) under a local orthonormal frame $e_1,\ldots, e_n$ of $\Sigma$:
\begin{eqnarray}\label{lap-a}
\Delta a_{ij}&=&\nabla^2H(e_j,e_i)+\sum_{k=1}^n\bar{R}_{\nu kik;j}+\sum_{k=1}^n\bar{R}_{\nu ijk;k}\\
&& -H\bar{R}_{\nu ij\nu}-\overline{\textrm{Ric}}_{\nu\nu}a_{ij}
+H\sum_{k=1}^na_{ik}a_{kj}-|A|^{2}a_{ij}\nonumber\\
&&+\sum_{k,l=1}^n(a_{ik}\bar{R}_{kljl}+a_{jk}\bar{R}_{klil}+2a_{kl}\bar{R}_{lijk}).\nonumber
\end{eqnarray}
Observe that
\begin{eqnarray}\label{eq-R-1}
\sum_{k=1}^{n}\bar{R}_{\nu kik;j}&=&(\bar{\textrm{Ric}})_{\nu i;j}=(\bar{\textrm{Ric}})_{i\nu;j},\\
\label{eq-R-2}\sum_{k=1}^{n}\bar{R}_{\nu ijk;k}
&=&\sum_{k=1}^{n}\bar{R}_{jk\nu i;k}=-\sum_{k=1}^{n}\bar{R}_{jkik;\nu}-\sum_{k=1}^{n}\bar{R}_{jkk\nu;i}\\
&=&-(\bar{\textrm{Ric}})_{ij;\nu}+\bar{R}_{i\nu j\nu;\nu}+(\bar{\textrm{Ric}})_{j\nu;i},\nonumber\\
\label{eq-R-3}
\sum_{k,l=1}^{n}a_{ik}\bar{R}_{kljl}&=&\sum_{k=1}^{n}a_{ik}(\bar{\textrm{Ric}})_{kj}-\sum_{k=1}^{n}a_{ik}\bar{R}_{k\nu j\nu},\\
\label{eq-R-4}
\sum_{k,l=1}^{n}a_{jk}\bar{R}_{klil}&=&\sum_{k=1}^{n}a_{jk}(\bar{\textrm{Ric}})_{ki}-\sum_{k=1}^{n}a_{jk}\bar{R}_{k\nu i\nu}.
\end{eqnarray}
\noindent Using the same local frame as in the proof of  Prop.\ref{prop1} and substituting \eqref{Hessian-H}, \eqref{eq-R-1}, \eqref{eq-R-2}, \eqref{eq-R-3} and \eqref{eq-R-4} into \eqref{lap-a}, we have that at $p$, for $1\leq i,j\leq n$,
\begin{eqnarray}\label{lap-a-2}
\Delta a_{ij}&=&(\overline{\textrm{Ric}}_{f})_{i\nu;j}+
(\overline{\textrm{Ric}}_{f})_{j\nu;i}-
(\overline{\textrm{Ric}}_{f})_{ij;\nu}+\bar{R}_{i\nu j\nu;\nu}\\
& &+\sum_{k=1}^na_{ik}(\overline{\textrm{Ric}}_{f})_{kj}+\sum_{k=1}^na_{jk}(\overline{\textrm{Ric}}_{f})_{ki}\nonumber\\
& &-\sum_{k=1}^na_{ik}\bar{R}_{k\nu j\nu}-\sum_{k=1}^na_{jk}\bar{R}_{k\nu i\nu}-2\sum_{l,k=1}^na_{lk}\bar{R}_{likj}\nonumber\\
& &+\sum_{k=1}^na_{ij,k}f_{k}-(\overline{\textrm{Ric}}_{f})_{\nu\nu}a_{ij}-|A|^{2}a_{ij}.\nonumber
\end{eqnarray}
Multiply \eqref{lap-a-2}  by  $a_{ij}$  and take the trace. Then it holds that at $p$,
\begin{eqnarray}\label{eq-simons-1}
\frac{1}{2}\Delta_f|A|^{2}&=&\frac{1}{2}\Delta|A|^2-\frac{1}{2}\langle\nabla f,\nabla|A|^2\rangle\nonumber \\
&=& \sum_{i,j=1}^{n}a_{ij}\Delta a_{ij}+\sum_{i,j,k=1}^{n}a_{ij,k}^2-\sum_{i,j,k=1}^na_{ij}f_ka_{ij,k}\nonumber \\
&=&|\nabla A|^{2}+2\sum_{i,j,k=1}^{n}a_{ij}a_{ik}(\bar{\textrm{Ric}}_{f})_{kj}-(\bar{\textrm{Ric}}_f)_{\nu\nu}|A|^{2}-|A|^{4}\\
& &+2\sum_{i,j=1}^{n}a_{ij}(\bar{\textrm{Ric}}_{f})_{i\nu;j}-\sum_{i,j=1}^{n}a_{ij}(\bar{\textrm{Ric}}_f)_{ij;\nu}\nonumber\\
&&+\sum_{i,j=1}^{n}a_{ij}\bar{R}_{i\nu j\nu;\nu}\nonumber\\
& &-2\sum_{i,j,k=1}^{n}a_{ij}a_{ik}\bar{R}_{k\nu j\nu}-2\sum_{i,j,k,l=1}^{n}a_{ij}a_{lk}\bar{R}_{iljk}.\nonumber
\end{eqnarray}
Thus
\begin{eqnarray}\label{simons-eq1}
\frac{1}{2}\Delta_f|A|^{2}&=&|\nabla A|^{2}+2\langle A^2, \bar{\textrm{Ric}}_{f}|_{\Sigma}\rangle-(\bar{\textrm{Ric}}_f)_{\nu\nu}|A|^{2}-|A|^{4}\nonumber\\
& &+2\langle A, \overline{\nabla} (\bar{\textrm{Ric}}_{f})(\cdot,\nu, \cdot)|_{\Sigma}\rangle-\langle A, \overline{\nabla} (\bar{\textrm{Ric}}_{f})(\nu, \cdot,\cdot)|_{\Sigma}\rangle\nonumber\\
&&+\langle A, \overline{\nabla} (\bar{Rm})(\nu, \cdot,\nu,\cdot,\nu)|_{\Sigma}\rangle\\
& &-2\langle A^2, \bar{Rm}(\cdot,\nu,\cdot,\nu)|_{\Sigma}\rangle -2\sum_{i,j,k,l=1}^{n}a_{ij}a_{lk}\bar{R}_{iljk}.\nonumber
\end{eqnarray}
Since \eqref{simons-eq1} is independent of the choose of the coordinates,     \eqref{simons-eq} holds on $\Sigma$.

\end{proof}
When the ambient space $M$ has the property: $\overline{Ric}_f=C\overline{g}$, i.e., $M$ is a gradient Ricci soliton, the Simons' type equation for $f$-minimal hypersurfaces are reduced to that
\begin{cor} \label{simons-eq-grad}Let $(M^{n+1},\overline{g}, e^{-f}d\mu)$ be a smooth metric metric space satisfying $\overline{Ric}_f=C\overline{g}$, where $C$ is a constant. If $(\Sigma,g)$ is an $f$-minimal hypersurface isometrically immersed in $M$, then it holds that on $\Sigma$
\begin{eqnarray}
\frac{1}{2}\Delta_f|A|^{2}&=&|\nabla A|^{2}+C|A|^{2}-|A|^{4}
+\sum_{i,j=1}^{n}a_{ij}\bar{R}_{i\nu j\nu;\nu}\\
& &-2\sum_{i,j,k=1}^{n}a_{ij}a_{ik}\bar{R}_{j\nu k\nu}-2\sum_{i,j,k,l=1}^{n}a_{ij}a_{lk}\bar{R}_{iljk},\nonumber
\end{eqnarray}
where the notation is the same as in Theorem \ref{prop2}.
\end{cor}

Finally, in this section,  we  prove the following identity for $L_f$ operator, which is useful for the study of the eigenvalues of $L_f$ and $L_f$-index of $f$-minimal hypersurfaces.
\begin{prop}\label{aut2}
Let $(M^{n+1},\overline{g},e^{-f}d\mu)$ be a smooth metric measure space and $X$  a parallel vector field on $M$. If $(\Sigma^{n},g)$ is an  $f$-minimal hypersurface  isometrically immersed in $M$, then the function $\alpha:\Sigma\rightarrow \mathbb{R}$ defined by $\alpha=\langle X,\nu\rangle$ satisfies
\begin{eqnarray}\label{6eq-1}
\Delta_f\alpha&=&\overline{\textrm{Ric}}_{f}(X,\nu)-|A|^{2}\alpha-(\overline{\textrm{Ric}}_{f})_{\nu\nu}\alpha,\\
\label{6eq}
L_{f}\alpha&=&\overline{\textrm{Ric}}_{f}(X,\nu),
\end{eqnarray}
where the notation is the same as in Theorem \ref{prop2}.
\end{prop}
\begin{proof}  Choose a local field of orthonormal frame $\{e_{1},\ldots,e_{n},e_{n+1}\}$ on $M$ as in the proof of Proposition \ref{prop1}. Then  for $1\leq i\leq n$,
\begin{eqnarray}
e_i\alpha&=&\langle X,\overline{\nabla}_{e_{i}}\nu\rangle\nonumber\\
&=&\sum_{j=1}^na_{ij}\langle X,e_{j}\rangle.
\end{eqnarray}
Note $\nabla_{e_i}e_j(p)=0$, $1\leq i,j\leq n$. The Hessian of $\alpha$ at $p$ is given by
\begin{eqnarray}
(\nabla^2\alpha)(e_k,e_i)=e_ke_i\alpha&=&\sum_{j=1}^na_{ij,k}\langle X,e_{j}\rangle+\sum_{j=1}^na_{ij}\langle X,\overline{\nabla}_{e_{k}}e_{j}\rangle\nonumber\\
&=&\sum_{j=1}^na_{ij,k}\langle X,e_{j}\rangle-\sum_{j=1}^na_{ij}a_{kj}\langle X,\nu\rangle\nonumber\\
&=&\sum_{j=1}^na_{ik,j}\langle X,e_{j}\rangle +\sum_{j=1}^n\bar{R}_{\nu ijk}\langle X,e_{j}\rangle-\sum_{j=1}^na_{ij}a_{jk}\alpha.\label{eig1}
\end{eqnarray}
Take the trace in \eqref{eig1}. Then
\begin{equation}\label{eig2}
\Delta\alpha=\langle\nabla H,X\rangle+\sum_{j=1}^n(\overline{\textrm{Ric}})_{\nu j}\langle X,e_{j}\rangle-|A|^{2}\alpha.
\end{equation}
Also from \eqref{der-H} and \eqref{eig1}, we have at $p$,
\begin{eqnarray*}
\langle\nabla H,X\rangle&=&\sum_{j=1}^n(\overline{\nabla}^2{f})_{\nu j}\langle X,e_{j}\rangle+\sum_{j,k=1}^na_{jk}f_{k}\langle X,e_{j}\rangle\\
&=&\sum_{j=1}^n(\overline{\nabla}^2{f})_{\nu j}\langle X,e_{j}\rangle+\sum_{k=1}^nf_{k}\alpha_{k}\\
&=&\sum_{j=1}^n(\overline{\nabla}^2{f})_{\nu j}\langle X,e_{j}\rangle+\langle\nabla f,\nabla\alpha\rangle.
\end{eqnarray*}
Substituting the above identity into \eqref{eig2}, we have
\begin{eqnarray}
\Delta\alpha&=&\sum_{j=1}^n(\overline{\textrm{Ric}}_{f})_{\nu j}\langle X,e_{j}\rangle+\langle\nabla f,\nabla\alpha\rangle-|A|^{2}\alpha\nonumber\\
&=&\overline{\textrm{Ric}}_{f}(X,\nu)+\langle\nabla f,\nabla\alpha\rangle-|A|^{2}\alpha-(\overline{\textrm{Ric}}_{f})_{\nu\nu}\alpha, \label{eig3}
\end{eqnarray}
that is 
\begin{eqnarray}
\Delta_f\alpha&=&\overline{\textrm{Ric}}_{f}(X,\nu)-|A|^{2}\alpha-(\overline{\textrm{Ric}}_{f})_{\nu\nu}\alpha.\label{eigen-1-18}
\end{eqnarray}
Since $p$ is arbitrary and~\eqref{eigen-1-18} is independent of the frame, we have proved \eqref{6eq-1} and then  \eqref{6eq}.

\end{proof}

When the ambient manifold is gradient Ricci soliton, we obtain that
\begin{cor}\label{cor-alpha}  Let $(M^{n+1},\overline{g}, e^{-f}d\mu)$ be a smooth metric metric space satisfying $\overline{Ric}_f=C\overline{g}$, where $C$ is a constant. Suppose that $X$ is a parallel vector field on $\Sigma$. If $(\Sigma,g)$ is an $f$-minimal hypersurface isometrically immersed in $M$, then $\alpha=\langle X,\nu\rangle$ satisfies that on $\Sigma$,
\begin{equation}
L_{f}\alpha=C\alpha.
\end{equation}
\end{cor}
\begin{example} \cite{CM3}
Let $M=\mathbb{R}^{n+1}$ and $f=\frac{|x|^{2}}{4}$. The $f$-minimal hypersurfaces are self-shrinkers.  Suppose that $\Sigma$ is a self-shrinker, then
\begin{eqnarray*}
\frac{1}{2}\Delta_fH^{2}&=&|\nabla H|^{2}+(\frac{1}{2}-|A|^{2})H^{2},\\
\frac{1}{2}\Delta_f|A|^{2}&=&|\nabla A|^{2}+(\frac{1}{2}-|A|^{2})|A|^{2},\\
L_{f}H&=&H.
\end{eqnarray*}
If $V$ is a constant vector in $\mathbb{R}^{n+1}$ and $\nu$ is the unit normal to $\Sigma$, then
\begin{equation*}\label{eig9}
L_f\langle V,\nu\rangle=\frac{1}{2}\langle V,\nu \rangle.
\end{equation*}
\end{example}


\section{$L_f$-index of  $f$-minimal hypersurfaces}\label{sec-index}

 In this section, we study the $L_{f}$-index of  closed  $f$-minimal hypersurfaces immersed in  the product manifold $\mathbb{S}^{n}(a)\times\mathbb{R}$, $n\geq 2$, with $f(x,t)=\frac{(n-1)t^2}{4a^2}$, where $(x,t)\in \mathbb{S}^{n}(a)\times\mathbb{R}$ and  $\mathbb{S}^{n}(a)$ denotes the round sphere in $\mathbb{R}^{n+1}$ of radius $a$. For simplicity of notation, we only consider $a=\sqrt{2(n-1)}$ and hence  $f=\frac{t^2}{4}$.  The cases of other $a$ are analogous.
$\mathbb{S}^{n}(\sqrt{2(n-1)})\times\mathbb{R}$ has the  metric
\begin{equation*}
\overline{g}=g_{\mathbb{S}^{n}(\sqrt{2(n-1)})}+dt^{2},
\end{equation*}
 where $g_{\mathbb{S}^{n}(\sqrt{2(n-1)})}$ denotes the canonical metric of  $\mathbb{S}^{n}(\sqrt{2(n-1)})$. 
Let $\{\overline{e}_{1},\ldots,\overline{e}_{n+1}\}$ be a local  orthonormal frame on $\mathbb{S}^{n}(\sqrt{2(n-1)})\times\mathbb{R}$.  By a straightforward computation, one has the components of the curvature tensor and Ricci curvature tensor of $\mathbb{S}^{n}(\sqrt{2(n-1)})\times\mathbb{R}$ are given by, for $1\leq i,j,k,l\leq n+1$,
\begin{eqnarray}
\overline{R}_{ijkl}&=&\frac{1}{2(n-1)}\biggr(\delta_{ik}\delta_{jl}-\delta_{il}\delta_{jk}-\langle \overline{e}_{j},\frac{\partial}{\partial t}\rangle\langle \overline{e}_{l},\frac{\partial}{\partial t}\rangle\delta_{ik} \nonumber\\
& &-\langle \overline{e}_{i},\frac{\partial}{\partial t}\rangle\langle \overline{e}_{k},\frac{\partial}{\partial t}\rangle\delta_{jl}+\langle \overline{e}_{j},\frac{\partial}{\partial t}\rangle\langle \overline{e}_{k},\frac{\partial}{\partial t}\rangle\delta_{il}+\langle \overline{e}_{i},\frac{\partial}{\partial t}\rangle\langle \overline{e}_{l},\frac{\partial}{\partial t}\rangle\delta_{jk})\biggr),\label{curva}
\end{eqnarray}
and
\begin{equation}\label{eig7}
(\overline{\textrm{Ric}})_{ik}=\frac{1}{2}\bigg(\delta_{ik}-\langle \overline{e}_{i},\frac{\partial}{\partial t}\rangle\langle \overline{e}_{k},\frac{\partial}{\partial t}\rangle\bigg).
\end{equation}
On the other hand, 
 $$\overline{\nabla}f=\frac t2\frac{\partial }{\partial t}.$$
\begin{equation}\label{hess-f}
(\overline{\nabla}^2f)_{ik}=\frac12\langle \overline{e}_{i},\frac{\partial}{\partial t}\rangle\langle \overline{e}_{k},\frac{\partial}{\partial t}\rangle.
\end{equation}
By \eqref{eig7} and \eqref{hess-f}, we have
\begin{equation}\label{ric-f}
(\overline{\textrm{Ric}})_{ik}+(\overline{\nabla}^2f)_{ik}=\frac{1}{2}\overline{g}_{ik}.
\end{equation}
Hence  $(\mathbb{S}^{n}(\sqrt{2(n-1)})\times\mathbb{R},\overline{g},e^{-f}d\mu)$ is a  smooth metric measure space with $\overline{\textrm{Ric}}_{f}=\frac12\overline{ g}$. In addition, in the theory of Ricci flow,  $(\mathbb{S}^{n}(\sqrt{2(n-1)})\times\mathbb{R},\overline{g},f)$ is a shrinking gradient solitons.

For an $f$-minimal hypersurface $\Sigma$ immersed  in $\mathbb{S}^{n}(\sqrt{2(n-1)})\times \mathbb{R}$,  
$$0=H_f=H-\frac t2\langle \frac{\partial}{\partial t},\nu\rangle=H-\frac{t}{2}\alpha,$$
 where $\alpha=\langle \frac{\partial}{\partial t},\nu\rangle$. So $\Sigma$ satisfies $$H=\frac{t\alpha}{2}.$$
 The operator $L_f$ on $\Sigma$ is that
\begin{equation}L_f=\Delta-\frac{t}{2}\langle \left(\frac{\partial}{\partial t}\right)^{T},\nabla\cdot\rangle+|A|^2+\frac{1}{2}.
\end{equation}

\begin{lemma}\label{ex-fmin-1} The slice $\mathbb{S}^{n}(\sqrt{2(n-1)})\times \{0\}$ is an $f$-minimal hypersurface in $\mathbb{S}^{n}(\sqrt{2(n-1)})\times \mathbb{R}$. Moreover a complete $f$-minimal hypersurface $\Sigma$ is immersed in a horizontal slice $\mathbb{S}^{n}(\sqrt{2(n-1)})\times\{t\}$, where  $t\in \mathbb{R}$ fixed,  if and only if $\Sigma$ is $ \mathbb{S}^{n}(\sqrt{2(n-1)})\times \{0\}$.
\end{lemma}
\begin{proof}   The unit normal $\nu$ to $\Sigma$ satisfies $\nu=\frac{\partial}{\partial t}$ and  hence $AX=\overline{\nabla}_X\nu=0, X\in T\Sigma$. Thus $\Sigma$ is totally geodesic.  Meanwhile
\begin{equation*}
\langle\overline{\nabla}f,\nu\rangle=\frac t2,
\end{equation*}
\begin{equation}\label{slice}
H_{f}=H-\langle\overline{\nabla}f,\nu\rangle=-\frac t2.
\end{equation}
It follows that $\Sigma$ is $f$-minimal if and only if $t=0$. Further, by Gauss equation we know $\Sigma$ has the constant positive section curvature and hence is closed. Since the closed $\Sigma$ has the dimension $n$ and $\mathbb{S}^n$ is simply-connected, $\Sigma$ must be $\mathbb{S}^{n}(\sqrt{2(n-1)})\times \{0\}$.

\end{proof}

We will  prove that
\begin{lemma}\label{index-S-n}
$L_f$-ind $ (\mathbb{S}^{n}(\sqrt{2(n-1)})\times\{0\})=1$.
\end{lemma}
\begin{proof} On $\mathbb{S}^{n}(\sqrt{2(n-1)})\times\{0\}$, we have
$\nabla f=(\overline{\nabla}f)^{\top}=0$, $|A|^{2}=0$. 
Hence, 
\begin{equation}
L_{f}=\Delta_{\mathbb{S}^{n}(\sqrt{2(n-1)})}+\frac12.
\end{equation}
 Thus the eigenvalues of $L_{f}$ are 
\begin{equation*}
\mu_{k}=\lambda_{k}-\frac12,
\end{equation*}
where 
$\lambda_{k}=\frac{k(k+n-1)}{2(n-1)}, k=0,1,\ldots$
are the eigenvalues of the Laplacian $\Delta_{\mathbb{S}^{n}(\sqrt{2(n-1)})}$.
Observe that
\begin{eqnarray*}
\mu_{0}&=&-\frac12,\\
\mu_{k}&>&0,\qquad \textrm{for all} \quad k\geq1,
\end{eqnarray*}
that is, $L_{f}$ has an unique negative eigenvalue with  multiplicity  one. Therefore the $L_f$-index of  $\mathbb{S}^{n}(\sqrt{2(n-1)})\times\{0\}$ is $1$.
\end{proof}

We will prove Theorem \ref{theo6}, which says that   $\mathbb{S}^{n}(\sqrt{2(n-1)})\times\{0\}$  is the unique closed $f$-minimal hypersurface of $L_f$-index one.

\noindent {\it Proof of Theorem \ref{theo6}}.  On $\Sigma$,
\begin{equation*}
\nabla t=(\overline{\nabla}t)^{\top}=\frac{\partial}{\partial t}-\langle\frac{\partial}{\partial t},\nu\rangle\nu.
\end{equation*}
So
\begin{equation}\label{height}
|\nabla t|^{2}=1-\langle\frac{\partial}{\partial t},\nu\rangle^{2}=1-\alpha^2.
\end{equation}
\noindent Since $\Sigma$ is closed, there is a point $p\in\Sigma$ such that $t(p)=\max_{\Sigma} t$ and $|\nabla t|(p)=0$. By equation \eqref{height}, we have
\begin{equation}\label{equ6}
0=|\nabla t|^{2}(p)=1-\alpha^{2}(p).
\end{equation}
Hence $\alpha(p)=\pm1$ and so $\alpha\not\equiv0$. Since  $\frac{\partial}{\partial t}$ is a parallel vector field on $\mathbb{S}^{n}(\sqrt{2(n-1)})\times\mathbb{R}$ and $\overline{\textrm{Ric}}_{f}=\frac12\bar{g}$, Proposition \ref{aut2} implies that
\begin{equation*}
L_{f}\alpha=\overline{\textrm{Ric}}_{f}(\frac{\partial}{\partial t},\nu)=\frac12\alpha.
\end{equation*}
Thus $\alpha$ is an eigenfunction of $L_f$ with eigenvalue  $-\frac12$ and this implies that $L_f$-$\textrm{ind}(\Sigma)\geq1$. 

Now we consider the equality case. Lemma \ref{index-S-n} says that $\mathbb{S}^{n}(\sqrt{2(n-1)})\times\{0\}$ has  $L_f$-index one.
Conversely, if $L_f$-$\textrm{ind}(\Sigma)=1$, then $-\frac12$ is the first eigenvalue. Then the corresponding eigenfunction $\alpha$ cannot change sign.  We may assume that $\alpha>0$. On the other hand,  $L_{f}\alpha=\frac12\alpha$ and $\overline{\textrm{Ric}}_f=\frac12\overline{g}$ imply
\begin{equation*}
\Delta_{f}\alpha+|A|^{2}\alpha=0.
\end{equation*}
Hence
\begin{equation}
\Delta_{f}\alpha\leq0.
\end{equation}
By the maximum principle,  $\alpha$ is constant on $\Sigma$. On the other hand, by ~\eqref{equ6}, there is a point $p\in\Sigma$ such that $\alpha(p)=\pm1$. Since  $\alpha$ is positive, $\alpha\equiv1$. Hence $\nabla t=0$ and thus $\Sigma$ is in a horizontal slice $\mathbb{S}^{n}(\sqrt{2(n-1)})\times \{t\}$. By Lemma \ref{ex-fmin-1},  $\Sigma$ must be $\mathbb{S}^{n}(\sqrt{2(n-1)})\times\{0\}$.

\qed

\section{Pinching theorem}\label{sec-appl}

First, we will  derive various identities including a Simons' type equation (see \eqref{delta-3}) for $f$-minimal hypersurfaces immersed in $\mathbb{S}^{n}(\sqrt{2(n-1)})\times\mathbb{R}$. Next, we apply them to obtain a pinching result for $f$-minimal hypersurfaces. We use the same notations as in Section \ref{sec-index}.
\begin{prop}\label{lem-1}
 Let  $\Sigma$ be an $f$-minimal hypersurface immersed in the product manifold $\mathbb{S}^{n}(\sqrt{2(n-1)})\times\mathbb{R}$ with $f=\frac{t^{2}}{4}$.  Then 
\begin{eqnarray}
\label{delta-1}\frac{1}{2}\Delta_f\alpha^{2}&=&|\nabla\alpha|^{2}-|A|^{2}\alpha^{2},\\
\label{delta-2}\frac{1}{2}\Delta_fH^{2}&=&|\nabla H|^{2}-(|A|^{2}+\frac12)H^{2}+\frac12\langle\nabla\alpha^{2},\nabla f\rangle,\\
\label{delta-3}\frac{1}{2}\Delta_f|A|^{2}&=&|\nabla A|^{2}+|A|^{2}(\frac12-|A|^{2})-\frac{1}{n-1}(|\nabla\alpha|^{2}-\alpha^{2}|A|^{2})\\
& &-\frac{1}{n-1}(\alpha^{2}f-\langle\nabla\alpha^{2},\nabla f\rangle).\nonumber
\end{eqnarray}
\end{prop}
\begin{proof} Choose  a local  orthonormal frame field $\{e_i\}_{i=1}^{n+1}$  for $M$ 
so that,   restricted to $\Sigma$,  $\{e_i\}_{i=1}^{n}$ are tangent  to $\Sigma$,  and $e_{n+1}=\nu$ is the unit normal to $\Sigma$.  Recall Proposition \ref{aut2} states that
\begin{equation}\label{eqs3}
\Delta_f\alpha=\overline{\textrm{Ric}}_f(\nu,\frac{\partial}{\partial t})-|A|^{2}\alpha-\overline{\textrm{Ric}}_f(\nu,\nu)\alpha.
\end{equation}

\noindent Substituting \eqref{ric-f} in~\eqref{eqs3}, we have
\begin{equation}\label{eqs5}
\Delta_f\alpha=-|A|^{2}\alpha.
\end{equation}
\eqref{eqs5} implies  \eqref{delta-1}:
$$
\frac{1}{2}\Delta_f\alpha^{2}=|\nabla\alpha|^{2}+\alpha\Delta_f\alpha\\
=|\nabla\alpha|^{2}-|A|^{2}\alpha^{2}.$$
Now we prove \eqref{delta-2}. Note $f=\frac14t^2$. $\overline{\nabla}^3 f=0$. This and Proposition \ref{prop1} yield
\begin{eqnarray}\label{delta-2-2}
\Delta_fH&=&2\sum_{i,j=1}^{n}a_{ij}(\overline{\nabla}^2 f)_{ij}-\overline{\textrm{Ric}}_{f}(\nu,\nu)H-|A|^{2}H.
\end{eqnarray}
 Substituting \eqref{hess-f} and \eqref{ric-f} into \eqref{delta-2-2},  we have
\begin{eqnarray*}
\Delta_fH&=&\sum_{i,j=1}^{n}a_{ij}\langle e_i,\frac{\partial}{\partial t}\rangle\langle e_j,\frac{\partial}{\partial t}\rangle-\frac12 H-|A|^{2}H\\
&=&\langle\nabla\alpha,\frac{\partial}{\partial t}\rangle-\frac12 H-|A|^{2}H
\end{eqnarray*}
Then, 
\begin{eqnarray*}
\frac{1}{2}\Delta_fH^{2}&=&|\nabla H|^{2}+H\Delta_fH\\
&=&|\nabla H|^{2}+H\langle\nabla\alpha,\frac{\partial}{\partial t}\rangle-(|A|^{2}+\frac12)H^{2}\\
&=&|\nabla H|^{2}+\frac{ t\alpha}{2}\langle\nabla\alpha,\frac{\partial}{\partial t}\rangle-(|A|^{2}+\frac12)H^{2}\\
&=&|\nabla H|^{2}+\frac12\langle\nabla\alpha^{2},\nabla f\rangle-(|A|^{2}+\frac12)H^{2}.
\end{eqnarray*}
In the above, we used $H=\frac{t\alpha}{2}$ and $\overline{\nabla}f=\frac{t}{2}\frac{\partial}{\partial t}$.
Thus \eqref{delta-2} holds.
 Finally we prove \eqref{delta-3}. Since $\mathbb{S}^n(\sqrt{2(n-1)})\times\mathbb{R}$ is a symmetric space, $\overline{\nabla}R=0$.  By  the  Simons' type equation (Corollary \ref{simons-eq-grad}), it holds that
\begin{eqnarray}
\frac{1}{2}\Delta_f|A|^{2}&=&|\nabla A|^{2}+|A|^{2}(\frac12-|A|^{2})\nonumber\\
& &-2\sum_{i,j,k=1}^{n}a_{ij}a_{ik}\bar{R}_{k\nu j\nu}-2\sum_{i,j,k,l=1}^{n}a_{ij}a_{lk}\bar{R}_{iljk}.\label{eq-simon-s}
\end{eqnarray}
Substituting the curvature tensors \eqref{curva} into \eqref{eq-simon-s} and computing directly, we obtain
\begin{eqnarray*}
\frac{1}{2}\Delta_f|A|^{2}&=&|\nabla A|^{2}+|A|^{2}(\frac12-|A|^{2})
 -\frac{1}{n-1}\biggr(H^{2}-\alpha^{2}|A|^{2}\\
 & &-2H\sum_{i,j=1}^{n}a_{ij}\langle e_i,\frac{\partial}{\partial t}\rangle\langle e_j,\frac{\partial}{\partial t}\rangle+\sum_{i,j,k=1}^{n}a_{ij}a_{ik}\langle e_j,\frac{\partial}{\partial t}\rangle\langle e_k,\frac{\partial}{\partial t}\rangle\biggr).
\end{eqnarray*}
Note that the function $\alpha$ satisfies
$\alpha_i=\displaystyle\sum_{j=1}^{n}a_{ij}\langle e_j,\frac{\partial}{\partial t}\rangle.$
Hence, 
\begin{eqnarray*}
\frac{1}{2}\Delta_f|A|^{2}&=&|\nabla A|^{2}+|A|^{2}(\frac12-|A|^{2})\\
& &-\frac{1}{n-1}\bigr(H^{2}-\alpha^{2}|A|^{2}-2H\langle\nabla\alpha,\frac{\partial}{\partial t}\rangle+|\nabla\alpha|^{2}\bigr)\\
&=&|\nabla A|^{2}+|A|^{2}(\frac12-|A|^{2})-\frac{1}{n-1}(|\nabla\alpha|^{2}-\alpha^{2}|A|^{2})\\
& &-\frac{1}{n-1}(H^{2}-2H\langle\nabla\alpha,\frac{\partial}{\partial t}\rangle).
\end{eqnarray*}
Using
$\bar{\nabla}f=\frac t2\frac{\partial}{\partial t}$
and $H=\frac{ t\alpha}{2}$,
we obtain \eqref{delta-3}:
\begin{eqnarray*}
\frac{1}{2}\Delta_f|A|^{2}
&=&|\nabla A|^{2}+|A|^{2}(\frac12-|A|^{2})-\frac{1}{n-1}(|\nabla\alpha|^{2}-\alpha^{2}|A|^{2})\\
& &-\frac{1}{n-1}(\alpha^{2}f-\langle\nabla\alpha^{2},\nabla f\rangle).
\end{eqnarray*}
\end{proof}
Proposition \ref{lem-1} implies the following equations:
\begin{lemma}\label{lem-A}
If $\Sigma$ is a closed orientable $f$-minimal hypersurface immersed in $M=\mathbb{S}^{n}(\sqrt{2(n-1)})\times\mathbb{R}$, then
\begin{eqnarray}
 &&\int_{\Sigma}|\nabla\alpha|^2e^{-f}-\int_{\Sigma}\alpha^2|A|^2e^{-f}=0,\label{alpha-eq}\\
&&-\int_{\Sigma}|\nabla H|^{2}e^{-f}+\int_{\Sigma}H^{2}|A|^{2}e^{-f}+\frac14\int_{\Sigma}\alpha^{2}(1-\alpha^{2})e^{-f}=0,\label{lem-H-eq}
\end{eqnarray}
\begin{equation}
\int_{\Sigma}|\nabla A|^{2}e^{-f}+\int_{\Sigma}|A|^{2}(\frac12-|A|^{2})e^{-f}-\frac{1}{2(n-1)}\int_{\Sigma}\alpha^{2}(1-\alpha^{2})e^{-f}=0.\label{lem-A-eq}
\end{equation}
\end{lemma}
\begin{proof} \eqref{alpha-eq} can be obtained by integrating \eqref{delta-1} directly. Now we prove \eqref{lem-H-eq}. Since
\begin{eqnarray*}
\Delta f&=&tr\nabla^2f=\sum_{i=1}^n[(\overline{\nabla}^2f)_{ii}-a_{ii}f_{\nu}]\\
&=&\frac12\sum_{i=1}^{n}\langle e_{i},\frac{\partial}{\partial t}\rangle^{2}-H{f}_\nu\\
&=&\frac12(1-\alpha^{2})-\langle\overline{\nabla}f,\nu\rangle^2,
\end{eqnarray*}
\begin{eqnarray}\label{eqs9}
\Delta_ff&=&\frac12(1-\alpha^{2})-\langle\nabla f,\nabla f\rangle-\langle\overline{\nabla}f,\nu\rangle^2\\
&=&\frac12(1-\alpha^{2})-|\overline{\nabla}f|^{2}\nonumber\\
&=&\frac12(1-\alpha^{2})-f\nonumber.
\end{eqnarray}
Integrating  \eqref{delta-2} and using \eqref{eqs9},  we obtain
 \begin{eqnarray*}
&&-\int_{\Sigma}|\nabla H|^{2}e^{-f}+\int_{\Sigma}(|A|^{2}+\frac12)H^{2}e^{-f}\\
&=&\frac12\int_{\Sigma}\langle\nabla\alpha^{2},\nabla f\rangle e^{-f}\\
&=&-\frac12\int_{\Sigma}\alpha^{2}(\Delta_{f}f)e^{-f}\\
&=&-\frac14\int_{\Sigma}\alpha^{2}(1-\alpha^{2})e^{-f}+\frac12\int_{\Sigma}H^{2}e^{-f}
\end{eqnarray*}
In the above we have used $\int_\Sigma\Delta_fH^2e^{-f}=0$ and $H^2=\alpha^2f$.   Thus, we get \eqref{lem-H-eq}. Finally we prove \eqref{lem-A-eq}.
Integrating \eqref{delta-3} and using \eqref{alpha-eq} and \eqref{eqs9}, we have
\begin{eqnarray*}
\int_{\Sigma}|\nabla A|^{2}e^{-f}+\int_{\Sigma}|A|^{2}(\frac12-|A|^{2})e^{-f}&=&\frac{1}{n-1}\int_{\Sigma}(\alpha^{2}f-\langle\nabla\alpha^{2},\nabla f\rangle)e^{-f}\\
&=&\frac{1}{n-1}\int_{\Sigma}(\alpha^{2}f+\alpha^{2}\Delta_ff)e^{-f}\\
&=&\frac{1}{2(n-1)}\int_{\Sigma}\alpha^{2}(1-\alpha^{2})e^{-f}.
\end{eqnarray*}

\end{proof}

Using Lemma \ref{lem-A},  we may prove Theorem \ref{pinching}.

\noindent {\it Proof of Theorem \ref{pinching}}. Observe that, for $n\geq 3$, 
\begin{eqnarray*}
&&|A|^{2}(\frac{1}{2}-|A|^{2})-\frac{1}{2(n-1)}\alpha^{2}(1-\alpha^{2})\\&=&-(|A|^2-\frac14)^2+(\frac{1}{4})^2[1-\frac{8}{n-1}\alpha^2(1-\alpha^2)]\geq 0\\
\end{eqnarray*}
if and only if 
 \begin{equation*}
 \frac14\biggr(1-\sqrt{1-\frac{8}{n-1}\alpha^2(1-\alpha^2)}\biggr)\leq |A|^2\leq \frac14\biggr(1+\sqrt{1-\frac{8}{n-1}\alpha^2(1-\alpha^2)}\biggr).
 \end{equation*}
So  \eqref{lem-A-eq} implies that  on $\Sigma$, for $n\geq3$, 
\begin{equation*}
\nabla A\equiv0,
\end{equation*}
and 
\begin{equation*}
|A|^{2}(\frac{1}{2}-|A|^{2})-\frac{1}{2(n-1)}\alpha^{2}(1-\alpha^{2})=0.
\end{equation*}
Hence, $|A|^{2}$ and $H$ are constants. Substituting in (\ref{lem-H-eq}), we obtain
\begin{equation*}
\int_{\Sigma}|A|^{2}H^{2}e^{-f}+\frac{1}{4}\int_{\Sigma}\alpha^{2}(1-\alpha^{2})e^{-f}=0.
\end{equation*}
So
\begin{equation*}
\alpha^{2}(1-\alpha^{2})=0.
\end{equation*}
This implies that on $\Sigma$,
\begin{equation*}
\alpha\equiv0\qquad\emph{or}\qquad\alpha^2\equiv1
\end{equation*}
Since  $\Sigma$ is closed, $\alpha^2\equiv1$. Without loss of generality,   we choose $\alpha\equiv 1$.  So $\Sigma$ is in a horizontal slice $\mathbb{S}^n(\sqrt{2(n-1)})\times\{t\}$. By Lemma \ref{ex-fmin-1}, we conclude that $\Sigma$ is $\mathbb{S}^{n}(\sqrt{2(n-1)})\times\{0\}$.

\qed

\end{document}